\documentclass[11pt]{amsart}
\usepackage[left=1.5in, right=1.5in]{geometry}
\usepackage{amssymb}

\usepackage{amscd}

\usepackage{graphicx,psfrag,epsfig,multirow}
\usepackage{caption}
\usepackage{subcaption}

\usepackage[all,cmtip]{xy}
\usepackage{amssymb,amsmath,amscd,amsthm,verbatim}
\numberwithin{equation}{section}

\usepackage[active]{srcltx}
\usepackage{geometry}
\usepackage{amssymb,amsmath,amscd,amsthm}
\usepackage{graphicx}
\usepackage{caption}
\usepackage{mathrsfs}
\newtheorem{Thm}{Theorem}
\newtheorem{Def}{Definition}
\newtheorem{Lm}{Lemma}[section]
\newtheorem{Prop}[Lm]{Proposition}

\newtheorem{Cor}[Lm]{Corollary}

\def\bdef{\begin{Def}}
\def\endef{\end{Def}}
\def\bthm{\begin{Thm}}
\def\ethm{\end{Thm}}
\def\bprop{\begin{Prop}}
\def\enprop{\end{Prop}}
\def\blm{\begin{Lm}}
\def\elm{\end{Lm}}
\def\bcor{\begin{Cor}}
\def\ecor{\end{Cor}}
\def\brm{\begin{Rem}}
\def\erm{\end{Rem}}
\def\bfig{\begin{picture}}
\def\efig{\end{picture}}
\def\be{\begin{eqnarray}}
\def\ee{\end{eqnarray}}
\def\beal{\begin{aligned}}
\def\enal{\end{aligned}}

\newcommand{\G}{{\rm G}}
\def\Z{\mathbb Z}
\def\T{\mathbb T}
\def\C{\mathbb C}

\def\R{\mathbb R}
\def\eps{\varepsilon}
\def\al{\alpha}

\def\dt{\delta}

\def\lb{\lambda}

\def\id{\mathrm{id}}

\def\sP{\mathscr P}
\def\sC{\mathscr C}
\def\cC{\mathcal C}
\def\sH{\mathscr H}

\def\cS{\mathcal S}
\def\cA{\mathcal A}

\def\cD{\mathcal D}

\def\FH{\mathrm{HF}}
\def\FC{\mathrm{CF}}
\newcommand{\iS}{\underleftarrow{\mathstrut{\rm SH}}}
\newcommand{\dS}{\underrightarrow{\mathstrut{\rm SH}}}
\title[Perioidic orbits of Hamiltonians separating Lagrangian tori]{Existence of noncontractible periodic orbits of Hamiltonian systems separating two Lagrangian tori on $T^*\T^n$ with application to non convex Hamiltonian systems}
\author{Jinxin Xue\\
University of Chicago}
\address{Department of mathematics, the University of Chicago, Chicago, IL, US, 60637}
\email{ jxue@math.uchicago.edu}%

\date\today
\begin{document}
\maketitle
\begin{abstract}
In this paper, we show the existence of non-contractible periodic orbits in Hamiltonian systems defined on $T^*\T^n$ separating two Lagrangian tori  under a certain cone assumption. Our result gives a positive answer to a question of Polterovich in \cite{P}. As an application, we find periodic orbits in almost all the homotopy classes on a dense set of energy levels in Lorentzian type mechanical Hamiltonian systems defined on $T^*\T^2$. This solves a problem of Arnold in \cite{A}.
\end{abstract}
\section{Introduction}
In a recent paper \cite{P}, Polterovich proved the existence of an invariant measure $\mu$ for Hamiltonian systems in the following setting. 

Consider a symplectic manifold $(M,\omega)$  and a pair of compact subsets $X, Y \subset M$ with the
following properties:
\begin{itemize}
\item [(P1)] $Y$ cannot be displaced from $X$ by any Hamiltonian diffeomorphism $\theta$:
\[\theta(Y) \cap X\neq  \emptyset \mathrm{\ for\ every\ }\theta\in \mathrm{Ham}(M, \omega).\]
\item [(P2)] There exists a path $\{\phi_t\}, t \in [0, 1],\ \phi_0 = \mathrm{id}$ of symplectomorphisms so that $\phi_1$ displaces $Y$ from $X: \phi_1(Y ) \cap X = \emptyset$.
\end{itemize}
We set $X':=\phi_1(Y)$ and $a:=\mathrm{Flux}(\{\phi_1\})=\int_0^1[i_{\dot\phi_t}\omega]\,dt$.
\begin{Thm}
[Theorem 1.1 of \cite{P}] \label{ThmPolt}
For every $F \in \sC^\infty(M,\R)$ with
\[F|_X \leq 0,\quad F|_{X'} \geq  1,\]
the Hamiltonian flow of $F$ possesses an invariant probability measure $\mu$ with
\begin{equation}\label{EqPolt}
\langle a,\rho(\mu)\rangle\geq 1
\end{equation} 
where $\rho(\mu)$ is the rotation vector of $\mu.$
\end{Thm}
A similar result is obtained in \cite{V} as Proposition 5.10 using a different approach.

It is natural to ask whether these invariant measures are supported on periodic orbits. In the same paper \cite{P}, the author asks the following question:

{\it Can one, under the assumptions of Theorem \ref{ThmPolt}, deduce existence of a closed orbit of the Hamiltonian flow so that the corresponding rotation vector satisfies inequality \eqref{EqPolt}? }
 
Finding periodic orbits is an important theme in symplectic dynamics. As remarked in \cite{G}, ``there is a general principle in symplectic dynamics that a compactly supported
function with sufficiently large variation must have fast non-trivial periodic orbits
or even one-periodic orbits if the function is constant near its maximum" (think of the Hofer--Zehnder capacity, for instance). However, this principle is not correct in full generality. There is a famous counter-example of Zehnder (cf. Example 2.1 of \cite{P}). On the manifold $(\T^4=\R^4/\Z^4,\omega= dp_1 \wedge dq_1 + \gamma dp_2 \wedge dq_1 + dp_2 \wedge dq_2)$ with $\gamma$ is irrational, the Hamiltonian $F(p, q) = \sin(2\pi p_1)$ carries no non-constant periodic orbits. Notice that $F$ separates the two Lagrangian tori $X=\{p=0\},\ X'=\left\{p=\left(1/2,0\right)\right\}$. 

So we specialize to the case \[M=T^*\T^n,\quad \omega_0=\sum_{i=1}^n dq_i\wedge dp_i\] and $X$ being the zero section $\{p=0\}$, $X'$ another Lagrangian torus corresponding to $\{p=p^*\neq 0\}$, and ask for the existence of noncontractible periodic orbits. However, there is an immediate counterexample given by the Hamiltonian \begin{equation}\label{EqCounter}
F(p,q)=\dfrac{\langle \al,p\rangle}{\langle\al,p^*\rangle}
\end{equation} where $\al\in \R^n$ is completely irrational and $\langle\al,p^*\rangle\neq 0$. Any composition $\sigma\circ F$ with $\sigma:\R\to\R$ smooth has no nontrivial periodic orbits. We choose $\sigma(x)=0$ for $x\leq 0$,\ $\sigma=1$ for $x\geq 1$ and monotone, then multiply $\sigma\circ F$ by a compactly supported function $\eta(p):\ \R^n\to \R$ that is 1 on a large open ball containing $p^*$. If $\eta$ decays sufficiently slowly outside the big ball, then the Hamiltonian system $\eta(p)\cdot \sigma\circ F(p,q)$ has no noncontractible 1-periodic orbits. 

In this paper, we show the existence of noncontractible periodic orbits of Hamiltonian systems in the following setting. Consider the symplectic manifold $(T^*\T^n, \omega_0)$. Consider two Lagrangian tori in $T^*\T^n$: $X$ being the zero section $\{p=0\}$ and $X'$ the section $\{p=p^*\}$, where $p^*=(p_1^*,\ldots, p_n^*)\in \R^n\setminus\{0\}$ is a constant vector. This construction fits into Polterovich's question as follows: Set $Y=X'$. Then $X$ is symplectically isotopic to $Y$ by the isotopy $\phi_t(q,p)=(q,tp^*),$ and Flux$(\{\phi_1\})=[p^*\,dq]$, while $Y=X$ is not displaceable from itself by a Hamiltonian isotopy by Gromov's theorem. 

Next given a matrix $A\in \mathrm{GL}(n,\R)$ whose columns are $v_1,v_2,\ldots,v_n$, consider the cone $\mathcal{C}$ positively spanned by $v_1,\ldots, v_n\in \R^n$, 
\[\mathcal{C}=\mathrm{span}_+\{v_1,v_2,\ldots,v_n\}:=\left\{\sum_{i=1}^n c_i v_i\ |\ c_i> 0,\quad \forall\ i=1,2,\ldots,n\right\}.\]
Denote its ``dual cone" $\mathcal C^*$ by 
\[\mathcal C^*:= \mathrm{span}_+\left\{\left(A^{T}\right)^{-1}e_1,\ \left(A^{T}\right)^{-1}e_2,\ldots,\ \left(A^{T}\right)^{-1}e_n \right\}\]
where $e_i,\ i=1,\ldots, n,$ are the standard basis vectors of $\R^n$. Both $\cC$ and $\cC^*$ are open. 

Our first result is the following. 
\begin{Thm}\label{ThmCorner}
Consider a cone $\mathcal C$ as above and its dual cone $\mathcal C^*$. Then for any point $p^*$ lying in the interior of the cone $\mathcal{C}$ the following hold: For every $\sC^\infty$ Hamiltonian $H(p,q,t)$ with compact support in $\cC\times \T^n\times \T$, and $H(p^*,q,t)\geq c>0$ for all $(q,t)\in\T^n\times\T$, and for every homology class $\al\in H_1(\T^n,\Z)\setminus \{0\}$ satisfying 
\begin{equation}\label{Eqal}\langle p^*,\al\rangle\leq c\quad \mathrm{and}\quad \al\in\mathcal C^*,\end{equation}
there exists a $1$-periodic orbit of $H$ in the homology class $\al$. 
\end{Thm}
A positively spanned cone cannot contain any line, so the cone $\mathcal{C}$ in the assumption is necessary to rule out the counterexample \eqref{EqCounter}. As the angle at the tip of the cone $\mathcal{C}$ becomes more obtuse, the set of homology classes for which Theorem \ref{ThmCorner} guarantees a $1$-periodic orbit becomes smaller. See Figure 1 for the picture of the cone $\mathcal C$ and its dual cone $\mathcal C^*$ in the two dimensional case (we choose $v_1=(-1,3),\ v_2=(3,-1)$).
\begin{figure}
       \centering
  \begin{subfigure}[b]{0.45\textwidth}
               \includegraphics[width=0.8\textwidth]{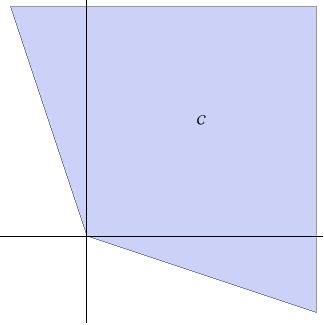}
                \caption{The cone $\mathcal C$
                 \
                 }
                
               \label{fig:profn1}
       \end{subfigure}%
             \begin{subfigure}[b]{0.45\textwidth}
                \includegraphics[width=0.8\textwidth]{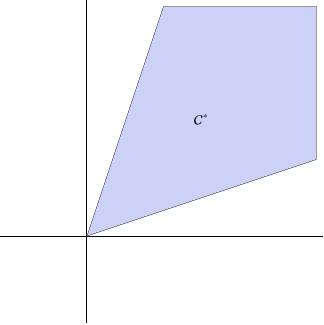}
                \caption{The dual cone $\mathcal C^*$}
                \label{fig:profn2}
       \end{subfigure}
       \caption{The cone $\mathcal C$ and the dual cone $\mathcal C^*$}\label{fig:cone}
\end{figure}

A closely related result is the following Theorem B of \cite{BPS}. To state the theorem, we first define the symplectic action as
 \begin{equation}\label{EqAction}
\mathcal{A}_H(x)=\int_0^1 (H(x(t),t)-\lambda(\dot x(t)))\,dt,\quad \mathrm{for\ } x\in \sC^\infty(\T^1,RT^*\T^n),
\end{equation}
where $\lb=p\,dq$ is the Liouville 1-form. 
 \begin{Thm}[Theorem B of \cite{BPS}]\label{ThmBofBPS}
For every Hamiltonian function $H(p,q,t)$ that is  compactly supported in $\{\|p\|<1\}\times \T^n\times\T^1$, and for every $e\in \Z^n$ such that
\[\Vert e\Vert\leq c:= \inf_{(q,t)\in\T^n\times \T^1} H(0,q,t),\]
the Hamiltonian system has a periodic solution $x$ in the homotopy class $e$ with action $\mathcal A_H (x) \geq c$.
\end{Thm}

We also get the following dense existence result.
\begin{Thm}\label{ThmDense}
Consider a $\sC^\infty$ autonomous Hamiltonian $H(p,q): T^*\T^n\to\R$ that is compactly supported in $\cC\times\T^n$, and satisfies $H(p^*,q)>0$ for some $p^*\in\cC$ and all $q\in \T^n$. Then for each nontrivial homology class $\al\in \cC^*\cap  H_1(\T^n,\Z)$,
 there exists a dense subset $S_\al\subset(0,\min_{q} H(p^*,q))$ with the property that for each $s\in S_\al$, the level set $\{H=s\}$ contains a closed orbit $($not necessarily of period 1$)$ in the class $\al$.
\end{Thm}

As an application of Theorem \ref{ThmCorner}, we answer a question of Arnold in the following Theorem \ref{ThmArnold}. For the problem, see Section 1.8 of \cite{A}, where Arnold asked for the existence of periodic orbits of the non-convex system $H=\frac{p_1^2}{2}-\frac{p_2^2}{2}+V(q_1,q_2),\quad (p,q)\in T^*\T^2,$ in each homology class. This system appears naturally when one wants to prove Arnold diffusion for non-convex type Hamiltonian systems and finding periodic orbits is the first thing one needs to do. Arnold remarked that ``It seems that the contemporary technique of the calculus of variation in the large has no ready methods for this problem", which seems still the case today. The next result shows the strength of our theorem when applied to nonconvex Hamiltonian systems. 
\begin{Thm}\label{ThmArnold}
Consider a Hamiltonian system of the form 
\begin{equation}\label{EqArnold}
H(p,q)=\frac{p_1^2}{2}-\frac{p_2^2}{2}+V(q_1,q_2),\quad (p,q)\in T^*\T^2,\quad V(q)\in \sC^\infty(\T^2,\R).
\end{equation} We normalize the potential $V$ by adding a constant such that $\max_q V(q)=0$. 
For each homology class $\al=(\al_1,\al_2)\in H_1(\T^2,\Z)\setminus\{0\}$ with $\al_1\neq \pm\al_2$, there exists a dense subset $S_\al$ of $(0,\infty)$ such that for each $s\in S_\al$, there exists a periodic orbit lying on the energy level $\{H=s\}$ with homology class $\al$. 
\end{Thm}
The idea is to notice that the function $p_1^2/2-p_2^2/2$ is positive in the interior of the cone spanned positively by $(1,1), (1,-1)$, and is zero on the boundary. After proper scaling and translation of the Hamiltonian to handle the bounded perturbation $V$, then composing it with $\sigma$ that we used in the paragraph of \eqref{EqCounter}, we get a modified Hamiltonian to which Theorem \ref{ThmCorner} is applicable. We obtain a periodic orbit lying on the energy level of the modified Hamiltonian, which is also a periodic orbit of the Hamiltonian \eqref{EqArnold}. See Section \ref{ssct:ThmArnold} for more details.

It is interesting to notice that the inequalities \eqref{EqPolt} in \cite{P,V} go in the opposite direction as ours \eqref{Eqal} (In our case, if we rescale the energy oscillation from $c$ to $1$, the corresponding time rescaling will take the rotation vector $\al$ to $\al/c$). The invariant measure $\mu$ found in \cite{V} verifies the equality (see Proposition 5.10 and Corollary 5.8 of \cite{V}) \begin{equation}\label{EqMather}
\mathcal{A}(\mu)=\boldsymbol{\al}(p^*)-\langle p^*,\rho(\mu)\rangle,
\end{equation} where $\boldsymbol\al(p^*)$ is Mather's $\boldsymbol\al$ function, which can be considered as energy $c$, and $\mathcal{A}$ is the symplectic action (see \cite{V} and our definition \eqref{EqAction}). On the other hand, in our proof, we always guarantee that our periodic orbits satisfy the inequality $\mathcal{A}\leq c-\langle p^*,\al\rangle$ (see Section \ref{ssctProof}). In Mather theory for positive definite Lagrangian systems \cite{M}, Equation \eqref{EqMather} implies that $\mu$ is action minimizing. So we may think that the invariant measures found by \cite{P,V} resemble the action minimizing measure of Mather. However, in our case, strict inequality may happen. Notice that our action has a negative sign compared to Mather's action. This shows that our periodic orbits may not be action minimizing in Mather's setting.

Let us now review the literature briefly. The existence of certain periodic orbits in Hamiltonian systems is part of the story of the Weinstein conjecture. We refer to \cite{G}
for a review. We focus on results mostly relevant to ours. In \cite{HV}, the authors prove the existence of periodic orbits for Hamiltonians separating neighbourhoods of two points on $\C \mathbb P^n$ using J-holomorphic curve techniques. 
Using the method of \cite{HV}, Gatien and Lalonde \cite{GL} showed the existence of noncontractible periodic orbits for compactly supported Hamiltonians separating two Lagrangian tori on $T^*K$ where $K$ is the Klein bottle as well as the case when $\Vert p^*\Vert$ is sufficiently small for $T^*\T^n$. In \cite{L}, Y.-J. Lee generalized the result of \cite{GL} by introducing a Gromov-Witten type invariant. Notice $T^*\T^n$ is exactly a case when the invariant of \cite{L} vanishes, so that we have counterexample \eqref{EqCounter} and the Gromov-Witten invariant approach does not work in out setting. On the other hand, there is a Floer theoretical approach developed in \cite{BPS}, the authors obtain several results including Theorem \ref{ThmBofBPS}. Their results are further generalized by \cite{W, SW} to general manifolds $T^*M$ where $M$ is closed. Other related results on the existence of non contractible periodic orbits are
obtained in [B, BH, G13, G14, GG2, N] etc. There is another important related
topic called the Conley conjecture, which asks for infinitely many periodic orbits. See [GG1] for an extensive review.

The proof of our results is to implement the machinery of \cite{BPS}.  We will show in the following sections that the method of \cite{BPS} goes through. 

The paper is organized as follows. In Section \ref{sctFloer}, we set up the machinery of Floer homology. This part follows mainly from \cite{BPS} with some variations following \cite{W}. Our new contribution is Lemma \ref{LmMorseBott}. We define the filtered Floer homology group in Section \ref{ssctFloer} and the inverse and direct limits of the groups in Section \ref{ssctLimits} induced by the monotone homotopies of Hamiltonians. We introduce exhausting sequences in Section \ref{ssctExhaust}, which reduce the computation of the Floer homology of any Hamiltonian to that for an exhausting sequence. In Section \ref{ssctCapacity}, we introduce a BPS type capacity which is suitable to find periodic orbits and another homological relative capacity which is accessible to computation and bounds the other capacity. Next, in Section \ref{ssctMorseBott}, we introduce the Morse-Bott theory which will be used to compute the Floer homology group of the exhausting sequence. In Section \ref{sctProof}, we prove Theorem \ref{ThmCorner}. In this section, we construct a family of profile functions giving an exhausting sequence and study their first and second order derivatives carefully. We use Morse-Bott theory to compute the Floer homology group for the profile functions. Finally in Section \ref{sctProofs}, we prove Theorems \ref{ThmCorner}, \ref{ThmDense} and \ref{ThmArnold}. 
\section{Floer homology}\label{sctFloer}
In this section, we set up the framework of \cite{BPS}. Since we specialize to the manifold $T^*\T^n$, we get some simplification in the presentation. 
\subsection{Floer theory and spectral invariants}\label{ssctFloer}
\subsubsection{Symplectic actions}
We consider the standard symplectic form $\omega_0=\sum_{i=1}^ndp_i\wedge dq_i$ and the Liouville 1-form $\lambda=\sum_{i=1}^np_idq_i$ such that $d\lambda=\omega_0.$ We choose some large $R$ (where $\Vert\cdot\Vert$ is the Euclidean norm) and denote by $RT^*\T^n$ the open set \[RT^*\T^n:=\left\{(p,q)\in T^*\T^n\ |\ \Vert p\Vert< R\right\}.\] 
In this section, we define Floer homology for functions in $\mathscr C^\infty_{cpt}(RT^*\T^n\times\T^1,\R)$, where $cpt$ means {\it compactly supported}.
We denote by $\mathscr L\T^n:=\mathscr C^\infty(\T^1,RT^*\T^n)$ the space of free loops of $RT^*\T^n$. To each $x\in \mathscr L\T^n$, we associate the free homotopy class $[x]\in \pi_1(\T^n)=H_1(\T^n,\Z)$. Given a homotopy class $\al\in \pi_1(\T^n)$ we denote its component by
\[\mathscr L_\al \T^n:=\{x\in \mathscr L\T^n\ |\ [x]=\al\}.\]
Each $H(p,q,t)\in \mathscr C^\infty_{cpt}(RT^*\T^n\times\T^1,\R)$ determines a vector field $X_H$ through 
\[i(X_H)\omega_0=-dH.\]
The space of 1-periodic solutions representing the class $\al$ is denoted by 
\[\mathscr{P}(H,\al):=\left\{x\in \mathscr L_\al \T^n\ |\ \dot x=X_H(x(t))\right\}.\]
Elements of $\mathscr{P}(H,\al)$ are the critical points of the symplectic action $\mathcal A_H(x)$ \eqref{EqAction} for $x\in \mathscr L_\al \T^n$.
\subsubsection{Action spectrum and periodic orbits}
The action spectrum is defined as
\[\mathcal{S}(H;\al)=\mathcal{A}_H(\mathscr{P}(H;\al))=\{\mathcal{A}_H(x)\ |\ x\in \mathscr P(H;\al)\}.\]
Consider $-\infty\leq a<b\leq \infty$ and denote by $\mathscr{P}^{[a,b)}(H;\al)$ the set of 1-periodic solutions of $H$ in the homotopy class $\al$ and with action in the interval $[a,b)$:
\[\mathscr{P}^{[a,b)}(H;\al):=\mathscr{P}^{b}(H;\al)\setminus \mathscr{P}^{a}(H;\al),\quad \mathscr{P}^a(H;\al):=\{x\in \mathscr{P}(H;\al)\ |\ \mathcal{A}_H(x)<a\}.\]
We assume that the Hamiltonian $H\in \mathscr C^\infty_{cpt}(RT^*\T^n\times\T^1,\R)$ satisfies the following nondegeneracy condition:

  \ \ \ \ {($\star$)\ \it $a,b\notin \cS(H;\al)$ and every $x\in \sP(H;\al)$ is nondegenerate in the sense that the derivative $d\phi_H^1(x(0))$ of the time-1 map $\phi^1_H$ does not have 1 in its spectrum. }

This nondegeneracy condition can be satisfied by perturbing $H$ near each periodic orbit (see Section 2.1 of \cite{W}). 
\subsubsection{Floer homology group}
We next define the Floer homology group $\FH^{[a,b)}$ with $\Z_2$ coefficients as the homology of the chain complex $\FC^{[a,b)}(H;\al)$ over $\Z_2$ which is generated by orbits in $\sP^{[a,b)}(H;\al)$
\[\FC^{[a,b)}(H;\al):=\FC^{b}(H;\al)/\FC^{a}(H;\al),\quad \FC^{a}(H;\al):=\bigoplus_{x\in \sP^a(H;\al)}\Z_2\, x.\]
\subsubsection{The boundary operator, energy}
To define the boundary operator, we consider the perturbed Cauchy-Riemann equation
\begin{equation}\label{EqCR}
\partial_s u+J_0(\partial_t u-X_H(u))=0,
\end{equation}
where $J_0=\left[\begin{array}{cc}
0&-\id_n\\
\id_n&0
\end{array}\right]$ is fixed in this paper.
 
We define the energy associated to a smooth solution $u(s,t):\ \R\times \T^1\to T^*\T^n$ of \eqref{EqCR} as
\[\mathcal{E}(u):=\int_{\T^1}\int_\R\Vert\partial_su\Vert^2\,ds\,dt.\]
If $u$ is a finite energy solution of \eqref{EqCR}, then we have
\begin{equation}\label{EqLimit}
\lim_{s\to\pm\infty} u(s,t)=x^\pm(t),\quad \lim_{s\to\pm\infty} \partial_su(s,t)=0
\end{equation}
and the convergences are uniform in $t$. Moreover, we have $x^\pm\in \sP(H;\al)$ and the energy identity
\begin{equation}\label{EqEnergy}\mathcal E(u)=\cA_H(x^-)-\cA_H(x^+).\end{equation}
\subsubsection{Compactness}\label{sssctCpt}
The energy identity \eqref{EqEnergy} and the exactness of $\omega_0$ imply that the space of finite energy solutions of \eqref{EqCR}, after quotient by the $\R$ action and compactification, is compact with respect to $\sC^\infty$ topology on compact sets. 
\subsubsection{Nondegeneracy}
 In order to guarantee the linearized operator of \eqref{EqCR} to be surjective, we do not perturb the almost complex structure $J_0$, instead, we perturb the Hamiltonian $H$ in an arbitrarily small neighborhood $U$ of the image of $u$ in $(R-\eps) T^*\T^n\times \T^1$ where $\eps$ is chosen to be so small that $H\in \sC^\infty_{cpt}((R-\eps)T^*\T^n\times\T^1,\R)$.
    We can choose the perturbation to vanish up to second order along the orbits $x^\pm$ (see Section 2.1 of \cite{W} and Theorem 5.1 (ii) of \cite{FHS}). Namely, there exists a small neighborhood $\mathscr{U}$ of zero in $\sC_{cpt}^\infty(U,\R)$ and a Baire's second category subset $\mathscr{U}_{reg}(\subset \mathscr U)$ of regular perturbations such that the linearized operator of (\ref{EqCR}) is surjective, for all $u$  solving the boundary value problem \eqref{EqCR} and \eqref{EqLimit} for the Hamiltonian $H'=H+h$ with $h\in\mathscr{U}_{reg}$. 
\subsubsection{Moduli space and Floer homology}
For every $H'=H+h$ with $h\in\mathscr U_{reg}$ and every pair of 1-periodic orbits $x^\pm\in\sP(H';\alpha)$, the moduli space $\mathscr{M}(x^-,x^+;H',J_0;\al)$ that is the space of solutions of the boundary value problem (\ref{EqCR}), (\ref{EqLimit}), is a smooth manifold of dimension $\mu_{CZ}(x^-)-\mu_{CZ}(x^+)$ near a
solution $u$, where $\mu_{CZ}$ is the Conley--Zehnder index.  It
follows from the compactness that the moduli space of index 1 modulo time shift, denoted by $\mathscr{M}^1(x^-,x^+;H',J_0;\al)/\R$, is a finite set for every pair $x^\pm\in\sP(H';\alpha)$ with $\mu_{CZ}(x^-)-\mu_{CZ}(x^+)=1$.  We next define the boundary operator $\partial^{H}$ as
\[\partial^{H}x := \sum_{y\in\sP^b(H;\alpha)} \sharp(\mathscr{M}^1(x,y;H',J_0;\al)/\R)\,y\]
for every $x\in\sP^b(H';\alpha)$ with $\mu_{CZ}(y)=\mu_{CZ}(x)+1$. The energy identity \eqref{EqEnergy} implies that $\FC^a(H;\alpha)$ is invariant under $\partial^H$. We then get a boundary operator $\left[\partial^{H}\right]:\ \FC^{[a, b)}(H;\alpha)\to \FC^{[a, b)}(H;\alpha)$ and the Floer homology
$\FH^{[a,b)}(H,\alpha) := \ker  \left[\partial^{H}\right]/{\rm
    im}\left[\partial^{H}\right].$ We remark that the Floer homology is defined independent of the choice of $h\in\mathscr U_{reg}$. 
\subsubsection{Homotopic invariance}
The above homology group $\FH^{[a,b)}(H,\alpha)$ is defined for a fixed Hamiltonian. When we have a smooth homotopy of Hamiltonians $H_s,\ s\in\R$ with $H_s=H_0$ when $s\leq 0$ and $H_s=H_1$ when $s\geq 1$,  we consider the following Cauchy-Riemann equation
\begin{equation}\label{EqCRs}
     \partial_su +J_0\partial_t  u -\nabla H_{s}(u,t)=0.
\end{equation}
The smooth solutions $u:\R\times \T^1\to RT^*\T^n$ of \eqref{EqCRs} are connecting orbits between two periodic orbits with the same Conley--Zehnder index. Namely we have uniformly in $t\in \T^1$
the limits \begin{equation*}\label{EqBC}
     \lim_{s\to-\infty} u(s,t)
     =x_0(t), \quad
     \lim_{s\to+\infty} u(s,t)
     =x_1(t), \quad
     \lim_{s\to\pm\infty} \partial_su(s,t)
     =0
\end{equation*}
where $x_i(t)\in \sP(H_i,\al),\ i=0,1$ and $\mu_{CZ}(x_0)=\mu_{CZ}(x_1)$.
We have the energy identity
\begin{equation*}\label{EqEnergys}
     \mathcal E(u) =\mathcal{A}_{H_0}(x_0)-\mathcal{A}_{H_1}(x_1)-\int_{-\infty}^{\infty}\int_{\T^1}\left(\partial_s H_{s}\right)(u(s,t),t) \, dt\, ds.
\end{equation*}
Similar to Section \ref{sssctCpt} we can find a second category subset of \emph{regular homotopies} among all homotopies such that  the linearized operator of (\ref{EqCRs}) is surjective, for all elements $u$ of the moduli spaces $\mathscr M(x_0,x_1;H_s,J_0;\alpha)$ (see also Section 2.1 of \cite{W}).   

Counting solutions of \eqref{EqCRs} defines the Floer chain map from $\FC(H_0,\al)$ to $\FC(H_1,\al)$. \subsubsection{Monotone homotopy}\label{sssctMono}    
Next, given $-\infty\leq a<b\leq\infty,$ we define \begin{equation}\label{EqHab}
\sH^{a,b}(RT^*\T^n;\al):=\left\{H\in \sC_{cpt}^\infty(RT^*\T^n\times\T^1,\R)\ |\ a,b\notin\cS(H,\al)\right\}
\end{equation}
as the set of compactly supported Hamiltonians whose action spectrum do not contain $a,b$. 
Suppose there are two Hamiltonians $H_0,H_1\in\mathscr H^{a,b}(RT^*\T^n;\alpha)$ satisfying
$$
H_0(p,q,t) \ge H_1(p,q,t),\quad \forall\ (p,q,t)\in RT^*\T^n\times \T^1
$$
 and that are nondegenerate in the sense of $(\star)$. Then there exists a
homotopy $s\mapsto H_s$ from $H_0$ to $H_1$ such that $ \partial_s H_s\le 0$, which is called a {\it monotone} homotopy.  A monotone homotopy gives rise to a
{\it monotone} homomorphism
\begin{equation}\label{EqHomo}
\sigma_{H_1H_0}:\FH^{[a,b)}(H_0;\alpha)\to \FH^{[a,b)}(H_1;\alpha),
\end{equation}
which is independent of the choice of the monotone homotopy of
Hamiltonians defining it. We have the composition rule
\begin{equation}\label{Eqcomp}
\sigma_{H_2H_1}\circ\sigma_{H_1H_0} = \sigma_{H_2H_0},
\end{equation}
for any $H_0,H_1,H_2\in \mathscr H^{a,b}(RT^*\T^n;\alpha)$ satisfy $H_0\ge H_1\ge
H_2$, and $\sigma_{HH}=\mathrm{id}$ for every $H \in \mathscr H^{a,b}(RT^*\T^n;\alpha)$.

To make the homomorphism $\sigma_{H_1H_0}$ an isomorphism, we need the following proposition.
\begin{Prop}[Proposition 4.5.1 of \cite{BPS}] \label{PropHomotopy}
   Consider $a<b,\ a,b\in[-\infty,\infty]$, a nontrivial homotopy class $\alpha\in \pi_1(T^*\T^n)$, and two Hamiltonians $H_0,H_1\in \mathscr H^{a,b}(RT^*\T^n;\alpha)$ with $H_0\geq H_1$ and with a monotone
   homotopy $\{H_s\}_{0\leq s\leq 1}$ satisfying
   $H_s\in \mathscr H^{a,b}(RT^*\T^n;\alpha)$, $\forall\ s\in[0,1]$.  Then the homomorphism
   $\sigma_{H_1H_0}$ in \eqref{EqHomo} 
   is an isomorphism. 
\end{Prop}
\subsection{Direct and inverse limits}\label{ssctLimits}
In this section, we introduce the direct and inverse limits of Floer homology groups. Our setting is slightly more general than needed to prove Theorem \ref{ThmCorner} but less general than that in \cite{BPS,W}. We denote by $V$ an open connected and bounded subset of $\R^n$, so $V\times\T^n$ is a subset of $RT^*\T^n$ for $R$ large enough. 
\subsubsection{Partial order on $\mathscr C^\infty_{cpt}(RT^*\T^n\times \T^1,\R)$}
We introduce a partial order on \\
$\mathscr C^\infty_{cpt}(RT^*\T^n\times \T,\R)$ defined by 
\[H_0\preceq H_1\quad\Leftrightarrow\quad  H_0(p,q,t)\geq H_1(p,q,t),\quad \forall\ (p,q,t)\in RT^*\T^n\times \T^1.\]
We get a partially ordered system $(\FH,\sigma)$ of $\Z_2$-vector spaces over $\sH^{a,b}(RT^*\T^n;\al)$ defined in \eqref{EqHab}. Namely, $\FH$ assigns to each $H\in\sH^{a,b}(RT^*\T^n;\al)$ the
$\Z_2$-vector space ${\FH}^{[a,b)}(H;\alpha)$, and $\sigma$ assigns to each pair $H_0\preceq H_1$ of $\sH^{a,b}(RT^*\T^n;\al)$ the monotone homomorphism
$ \sigma_{H_1H_0}$
with the composition rule (\ref{Eqcomp}).
\subsubsection{Inverse limit}\label{sssctInverseLim}
We restrict the partially ordered system $(\sH^{a,b}(RT^*\T^n;\al),\preceq)$ to a partially ordered system  $(\sH^{a,b}(V;\al),\preceq)$ where we define 
\begin{equation*}
\sH^{a,b}(V;\al):=\sH^{a,b}(RT^*\T^n;\al)\cap \sC^\infty_{cpt}(V\times \T^n\times \T^1,\R).
\end{equation*}
The partial order system $(\sH^{a,b}(V;\al),\preceq)$ is {\it downward directed}: For all $H_1,H_2\in\sH^{a,b}(V;\al)$ there exists $H_0\in\sH^{a,b}(V;\al)$ such that $H_0\preceq H_1$ and $H_0\preceq H_2$. The functor $(\FH,\sigma)$ is an \emph{inverse system of $\Z_2$-vector spaces over $\sH^{a,b}(V;\al)$} which has an {\it inverse limit}:
\begin{equation*}
\begin{aligned}
    \underleftarrow{\mathrm{SH}}^{[a,b)}(V;\alpha)  &:={\underleftarrow{\lim}}
     \FH^{[a,b)}(H;\alpha) \\
    &:=\left\{\left.\{a_H\}
     \in\prod_H
     \FH^{[a,b)}(H;\alpha)\ \right|\ 
     H_0\preceq H_1 \Rightarrow
     \sigma_{H_1H_0}(a_{H_0})=a_{H_1}\right\},
\end{aligned}
\end{equation*}
where $H,H_0,H_1\in\sH^{a,b}(V;\al)$. Next denote by
\begin{equation*}\label{EqpiH}
     \pi_H:\    \underleftarrow{\mathrm{SH}}^{[a,b)}(V;\alpha)  \to  \FH^{[a,b)}(H;\alpha)
\end{equation*}
 the projection to the component corresponding to $H\in\sH^{a,b}(V;\al)$. We have $$\pi_{H_1}=\sigma_{H_1H_0}\circ\pi_{H_0},\quad \mathrm{if\ }H_0\preceq H_1.$$
\subsubsection{Direct limit}\label{sssctDirectLim} Next, we fix $c>0$ and a point $p^*\in V$ and consider the subset
\begin{equation}\label{EqHabc}
     \sH^{a,b}_{c}(V,p^*;\al):=\left\{H\in \sH^{a,b}(V;\al)\ |\   H (p^*,q,t)>c\right\}.
\end{equation}
This set is {\it upward directed}. Namely, for all $H_0,H_1\in\sH^{a,b}_{c}(V,p^*;\al)$ there exists $H_2\in\sH^{a,b}_{c}(V,p^*;\al)$ such that $H_0\preceq H_2$ and
$H_1\preceq H_2$. The functor $(\FH,\sigma)$ is a {\it direct system of $\Z_2$-vector spaces over $\sH^{a,b}_{c}(V,p^*;\al)$}, whose \emph{direct limit} is defined as
\begin{equation*}
\begin{aligned}
     \underrightarrow{\mathrm{SH}}^{[a,b);c}(V,p^*;\alpha)
    &:=
     {\underrightarrow{\lim}}
     \FH^{[a,b)}(H;\alpha) \\
    &:=\left\{ (H,a_H)\ \left|\ 
     a_H\in \FH^{[a,b)}(H;\alpha)\right.\right\}
     \Big/\sim,
\end{aligned}
\end{equation*}
where $H\in\sH^{a,b}_{c}(V,p^*;\al),$ and the equivalence relation is defined as follows, $(H_0,a_{H_0})\sim(H_1,a_{H_1})$ if and only if there exists
$H_2\in\sH^{a,b}_{c}(V,p^*;\al)$ such that $H_0\preceq H_2$, $H_1\preceq H_2$ and $\sigma_{H_2H_0}(a_{H_0})=\sigma_{H_2H_1}(a_{H_1})$.
The direct limit is a
$\Z_2$-vector space with the operations
$$
     k[H_0,a_{H_0}]:=[H_0,ka_{H_0}],\quad  [H_0,a_{H_0}]+[H_1,a_{H_1}]:=[H_2,\sigma_{H_2H_0}(a_{H_0})+\sigma_{H_2H_1}(a_{H_1})],
$$
for all $k\in\Z_2$ and $H_2\in\sH^{a,b}_{c}(V,p^*;\al)$ such that $H_0\preceq H_2$ and $H_1\preceq H_2$. We define the homomorphism
\begin{equation*}\label{EqiH}
     \iota_H: \  \FH^{[a,b)}(H;\alpha)\to  \underrightarrow{\mathrm{SH}}^{[a,b);c}  (V,p^*;\alpha),\qquad   a_H\mapsto [H,a_H].
\end{equation*}
We have $\iota_{H_0}=\iota_{H_1}\circ \sigma_{H_1H_0}$, if $H_0\preceq H_1$.
\subsection{Exhausting sequence}\label{ssctExhaust}
We next introduce the notion of
exhausting sequences following \cite{BPS}.  Let $(\G,\sigma)$ be a partially ordered system
of $R$-modules over a partially ordered set $(I,\preceq)$ and denote $
\Z^\pm:=\left\{\nu\in\Z\,|\,\pm\nu>0\right\}.  $  We say a sequence
$\{i_\nu,\ \nu\in\Z^+\}$ (resp. $\{i_\nu,\ \nu\in\Z^-\}$) is {\it upward exhausting} (resp. {\it downward exhausting}) for
$(\G,\sigma)$, if
\begin{itemize}
  \item for all $\nu\in\Z^+$ (resp. $\{i_\nu,\ \nu\in\Z^-\}$), we have $i_\nu\preceq i_{\nu+1}$ (resp. $i_{\nu-1}\preceq
   i_\nu$) and that
   $\sigma_{i_{\nu+1}i_\nu}:\G_{i_\nu}\to\G_{i_{\nu+1}}$ (resp. $\sigma_{i_\nu i_{\nu-1}}:\G_{i_{\nu-1}}\to\G_{i_\nu}$) is an
   isomorphism;
  \item for all $i\in I$, there exists a $\nu\in\Z^+$ with $i\preceq i_\nu$ (resp. $\nu\in\Z^-$ with $i_\nu\preceq i$).
\end{itemize}
The next lemma shows how we can compute direct and inverse limits using exhausting sequences.
\begin{Lm}[Lemma 4.7.1 of \cite{BPS}] \label{LmExhaust}
   Consider $(\G,\sigma)$ a partially ordered system of $R$-modules over a partially ordered set  $(I,\preceq)$.
 Suppose $\{i_\nu,\ \nu\in\Z^+\}$ is upward exhausting $($resp. $\{i_\nu,\ \nu\in\Z^-\}$ is downward exhausting$)$ for
   $(\G,\sigma)$, then we have that the
   homomorphism $ \iota_{i_\nu}:\G_{i_\nu}\to\varinjlim\G $ $($resp. $ \pi_{i_\nu}:\varprojlim\G\to\G_{i_\nu} $$)$ is an
   isomorphism for all $\nu\in\Z^+$.
\end{Lm}
\subsection{Capacities}\label{ssctCapacity}
In this section, we introduce two capacities. The first one, called homological relative capacity, is defined by the existence of non vanishing homomorphism between the direct and inverse limits. The second one, called the symplectic relative capacity, is defined by the existence of certain periodic orbits. 
\subsubsection{Symplectic homology}
We cite the following proposition from \cite{BPS}.
\begin{Prop}[Proposition 4.8.1 of \cite{BPS}] \label{PropFactor}
   Consider $a<b,\ a,b\in[-\infty,\infty]$, and a nontrivial homotopy class $\alpha\in \pi_1(\T^n)$.  Then, for any
   $c\in\R$, there exists a unique homomorphism
   $$
   T_\alpha^{[a,b);c}:\iS^{[a,b)}(V;\alpha)\to\dS^{[a,b);c}(V,p^*;\alpha)
   $$
   such that for every $H\in\mathscr H^{a,b}_c(V,p^*;\alpha)$, the following diagram commute
   \[
   \xymatrix{ {\iS^{[a,b)}(V;\alpha)} \ar[rr]^-{T_{\alpha}^{[a,b);c}}
     \ar[dr]_-{\pi_{_H}}
     && {\dS^{[a,b);c}(V,p^*;\alpha)} \\
     & \FH^{[a,b)}(H;\alpha) \ar[ur]_-{\iota_{_H}} },
   \] 
 where $\pi_{_{H}}$ and $\iota_{_{H}}$
   are the homomorphisms in
   Section \ref{sssctInverseLim} and \ref{sssctDirectLim}.
\end{Prop}
\subsubsection{The homological relative capacity}
Following \cite{BPS} we define two capacities.
For every nontrivial homotopy class $\alpha\in\pi_1(\T^n)$
and real number $c>0$ we first define the set
$$
\mathscr A_c(V,p^*;\alpha) := \left\{a\in\R \,\Big|\ T_\alpha^{[a,\infty);c} \mathrm{\ in\ Proposition\ \ref{PropFactor}\ does\ not\ vanish} \right\}.
$$
 We next define the {\it
  homological relative capacity} of $(V,p^*)$ as the function
$$
\widehat{C}(V,p^*): \pi_1(\T^n)\times[-\infty,\infty)\to[0,\infty]
$$
\begin{equation}\label{EqCapacityHat}
(\al,a)\mapsto\widehat{C}(V,p^*;\alpha,a) := \inf\left\{c > 0\ |\ \sup
   \mathscr A_c(V,p^*;\alpha) > a \right\}.
\end{equation}
Here we use the convention that $\inf\emptyset = \infty$ and $\sup
\emptyset = -\infty$. 
\subsubsection{A relative symplectic capacity}
We define the BPS type relative symplectic capacity as
\begin{equation}\label{EqCap2}
\begin{aligned}
&C(V,p^*;\al,a):=\inf\left\{c>0\ |\ \forall\ H\in \mathscr{H}_c(V,p^*),\ \right. \\
&\left.\exists\ x\in \mathscr{P}(H;\al)\ \mathrm{such\ that}\ \mathcal{A}_H(x)\geq a\right\},
\end{aligned}
\end{equation}
where 
\begin{equation}\begin{aligned}\label{EqHamV}\mathscr H_c(V,p^*):&=\left\{H\in \mathscr C^\infty(RT^*\T^n\times\T^1,\R)\ |\ \right.\\
&\left. H(p^*,q,t)\geq c,\ \mathrm{and}\ H(p,q,t)=0\ \mathrm{for\ }p\in \R^n\setminus V\right\}.\end{aligned}\end{equation}
We get the existence of periodic orbits if we can bound $C(V,p^*;\al,a)$ from above.
\begin{Prop}[Proposition 4.9.1 of \cite{BPS}] \label{PropCap}
  Consider a number  $a\in\R$ and a nontrivial homotopy class $\alpha\in \pi_1(\T^n)$.  Suppose
   $\widehat{C}(V,p^*;\alpha,a) < \infty$ then every
   Hamiltonian $H\in \sH_c(V,p^*)$ with $c\ge\widehat{C}(V,p^*;\alpha,a)$ has a $1$-periodic orbit in the
   homotopy class $\alpha$ with action $\mathcal A_H(x)\geq a$.  In
   particular, we have
   $$
   \widehat{C}(V,p^*;\alpha,a) \ge C(V,p^*;\alpha,a).
   $$
\end{Prop}
The proof of this proposition is a word by word translation of that of Proposition 4.9.1 of \cite{BPS}. We remark here that the function class $\sH^{a,b}_c(V,p^*)$ \eqref{EqHabc} (with $b=\infty$) in the definition of $\widehat C(V,p^*;\al,a)$ \eqref{EqCapacityHat} differs from the function class
$\sH_c(V,p^*)$ in the definition of $C(V,p^*;\al,a)$ \eqref{EqCap2} for two reasons. First functions in $\sH_c(V,p^*)$ may have $a$ or $b$ in its action spectrum. Second, functions in $\sH_c(V,p^*)$ may not be compactly supported in $V\times \T^n\times \T$. However, functions in $\sH_c(V,p^*)$ can be approximated by that in $\sH^{a,b}_c(V,p^*)$. See the proof of Proposition 4.9.1 of \cite{BPS} for the approximation argument. 
\subsection{Morse-Bott theory in Floer homology}\label{ssctMorseBott}
We need to use Morse-Bott theory to compute Floer homology for Hamiltonians of the form $H(p)$. We first give the definition of Morse-Bott manifolds. 
\begin{Def}\label{DefMorseBott}
We say a subset $P\subset\sP(H;\al)$ is a Morse-Bott manifold of periodic orbits, if the set $C_0:=\{x(0)\ |\ x\in P\}$ is a compact submanifold of the symplectic manifold $M$, and the tangent space $$T_{x_0}C_0=\mathrm{Ker}(D\phi^1_H(x_0)-\id),\quad\forall\ x_0\in C_0,$$ where $\phi^1_H$ is the time-1 map of the Hamiltonian flow of $H(p,q,t)\in \sC^\infty_{cpt}(M\times \T^1,\R).$
\end{Def}
For a compactly supported Hamiltonian system $H(p)$ defined on $(RT^*\T^n,\omega_0)$ and depending only on variables in the fibers, the set $\{\frac{\partial H}{\partial p}(p),\ p\in \R^n\}\times \T^n$ is foliated into invariant tori labeled by frequencies $\{\frac{\partial H}{\partial p}(p),\ p\in \R^n\}$ according to the Liouville-Arnold theorem. A torus corresponding to a frequency $\dot q=\frac{\partial H}{\partial p}(p_0)\in \Z^n\setminus\{0\}$ is an invariant torus foliated by periodic orbits of period 1. If we pick any point $q(0)$ in the torus as initial condition to solve our Hamiltonian equation, the resulting periodic orbit lies completely on the torus. We have the following easy criteria to determine when such a torus is a Morse-Bott manifold. 
\begin{Lm}\label{LmMorseBott}For a Hamiltonian system $H(p)$ defined on $(T^*\T^n,\omega_0)$ and depending only on variables in the fibers, the set \[P=\left\{(p,q)\in T^*\T^n\ \Big|\ p=p_0,\quad \dot q=\frac{\partial H}{\partial p}(p_0)\in \Z^n\setminus \{0\}\right\}\] is a 
Morse-Bott manifold of periodic orbits for $H(p)$ if and only if \[\det\left(\dfrac{\partial^2 H}{\partial p_i\partial p_j}\right)(p_0)\neq 0,\ i,j=1,2,\ldots, n.\]
\end{Lm}
\begin{proof}
The Hamiltonian equations are $\begin{cases} \dot p=0\\ \dot q=\frac{\partial H}{\partial p}(p)\end{cases}.$
The linearized equation has the form 
\[\dfrac{d}{dt}\left[\begin{array}{c}\dt p\\
\dt q\end{array}\right]=\left[\begin{array}{cc}0&0\\
\frac{\partial^2 H}{\partial p^2}&0\end{array}\right]\left[\begin{array}{c}\dt p\\
\dt q\end{array}\right],\quad (\dt p,\dt q)\in T_{(p,q)}(T^*\T^n).\]
This equation can be integrated explicitly, whose fundamental solution at time $1$ is 
\[D\phi^1_H=\exp\left[\begin{array}{cc}0&0\\
\frac{\partial^2 H}{\partial p^2}&0\end{array}\right]=\id_{2n}+\left[\begin{array}{cc}0&0\\
\frac{\partial^2 H}{\partial p^2}&0\end{array}\right].\]
According to Definition \ref{DefMorseBott}, we only need to check
\[T_{(p_0,q_0)}P=\mathrm{Ker}(D\phi^1_H(p_0,q_0)-\id_{2n}),\quad \forall\ (p_0,q_0)\in P.\]
On the one hand, the set $P$ in consideration is an $n$-torus $P=\{(p,q)\ |\ p=p_0,\ q\in \T^n\}$, whose tangent space at $(p,q)$ is  $T_{(p,q)}P=\{(\dt p,\dt q)\ |\ \dt p=0,\ \dt q\in \R^n\}$.  On the other hand, we have \[\mathrm{Ker}(D\phi^1_H(p_0,q_0)-\id_{2n})=\mathrm{Ker}\left[\begin{array}{cc}0&0\\
\frac{\partial^2 H}{\partial p^2}&0\end{array}\right].\]
Hence $P$ is a Morse-Bott manifold if and only if the matrix $\frac{\partial^2 H}{\partial p^2}$ is nondegenerate at $p_0$. 
\end{proof}
Next, we cite a theorem of Pozniak from \cite{BPS} that computes Floer homology using a Morse-Bott manifold. 
\begin{Thm}[Theorem 5.2.2 of \cite{BPS}]\label{ThmMB}
   Consider $a<b,\ a,b\in [-\infty,\infty]$ and a nontrivial homotopy class $\alpha\in \pi_1(M)$, where $M$ is a symplectic manifold, and a Hamiltonian 
   $H$ which does not have $a$ or $b$ in its action spectrum. Suppose that the set
   $P:=\left\{x\in\mathscr P(H;\alpha)\,|\,a<\mathcal A_H(x)<b\right\}$ is a
   connected Morse--Bott manifold of periodic orbits.  Then we have
   $\FH^{[a,b)}(H;\alpha) \cong H_*(P;\Z_2)$.
\end{Thm}
\section{Construction of the profile functions}
In this section, we construct the key ingredient that we call the profile functions $H_s(p),\ s\in \R,$ needed in the proof of Theorem \ref{ThmCorner}. The family of profile functions is both upward and downward exhausting, and for $a$ satisfying $0\leq a\leq c-\langle p^*,\al\rangle$, all the homology groups $\FH^{[a,\infty)}(H_s,\al)$ are isomorphic to each other and nonvanishing as $s$ varies in $\R$. 

We first define an open sector region in $\R^n$ for some large $R$ as the set $V$ $$\mathcal C_{\id,R}:=\{p\in \R^n\ |\ \Vert p\Vert<R,\quad p_i>0,\ \forall\ i=1,2,\ldots,n\}.$$
For a given non degenerate matrix $\mathbf A\in$GL$(n,\R)$, we define 
$$\mathcal C_{\mathbf A,R}=\{p\in \R^n\ |\ \mathbf A^{-1}p \in \mathcal C_{\id,R}\},$$
which is a subset (containing a neighborhood of the tip) of the cone $\cC$ determined by the matrix $\mathbf A$.
We next define for a given point $p^*\in \cC_{\mathbf A,R}$ 
\begin{equation}
\begin{aligned}\label{EqHcHat}
&\widehat{\sH}_{c}(\cC_{\mathbf A,R},p^*):=\{H\in\sC^{\infty}_{cpt}(\cC_{\mathbf A,R}\times \T^n\times \T^1,\R)\ |\ H(p^*,q,t)>c\}.
\end{aligned}
\end{equation}We have the following list of requirements for the family of profile functions $H_s$.
\begin{itemize}
\item At the $\mathscr C^0$ level: $H_s(p)$ is both upward and downward exhausting in\\ $\widehat{\sH}_{c}(\cC_{\mathbf A,R},p^*)$, i.e. for each $H\in \widehat{\sH}_{c}(\cC_{\mathbf A,R},p^*)$
 there exist $s'<s$ such that $H_{s'}<H<H_s$. Functions in \eqref{EqHamV} can be approximated by functions in $\widehat{\sH}_{c}(\cC_{\mathbf A,R},p^*)$ defined here.  
\item At the $\mathscr C^1$ level: there exists a unique $p_s$ such that $DH_s(p_s)=\al$, the homology class in Theorem \ref{ThmCorner}, and the action of the corresponding periodic orbit is greater than $a\in [0, c-\langle p^*,\al\rangle]$. 
\item At the $\mathscr C^2$ level: $\det D^2 H_s(p_s)\neq 0,\quad \forall\ s,$ so that Lemma \ref{LmMorseBott} is satisfied. 
\item Monotonicity: $\partial_s H_s\geq 0.$
\end{itemize}
\subsection{A model function in the one dimensional case}
\subsubsection{Construction of a $\sC^1$ model function with parameter $s=1$}\label{SSSModel}
Consider the function $e^{-\frac{|x|^2}{2\dt}}$ where $\dt$ is sufficiently small. The second order derivative vanishes at the point $x=\pm\sqrt{\dt}$ (the turning points) where the first order derivative is $\mp\frac{1}{\sqrt{\dt}}e^{-\frac{1}{2}}$ and the value of the function is $e^{-1/2}\simeq 0.61>1/2$. 
\begin{figure}
        \centering
    \begin{subfigure}[b]{0.45\textwidth}
                \includegraphics[width=\textwidth]{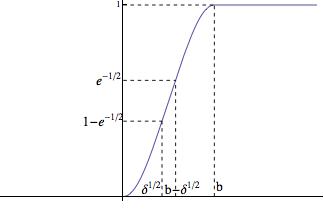}
               \caption{The model \\
                function
                 $u$\\
                 }
               \label{fig:profn1}
       \end{subfigure}%
             \begin{subfigure}[b]{0.55\textwidth}
                \includegraphics[width=\textwidth]{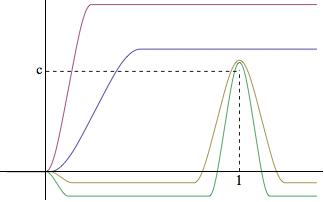}
                \caption{The profile functions seen from a section: $F_s(t\mathbf 1)$ as a function of $t$}
                \label{fig:profn2}
       \end{subfigure}
       \caption{Profile functions}\label{fig:profn}
\end{figure}

We define a $\mathscr C^1$ function $\hat u$ as follows. We consider one copy of $e^{-\frac{|x|^2}{2\dt}}$ and one copy $-e^{-\frac{|x|^2}{2\dt}}+1$. After suitably shifting horizontally the first function we use a piece of straight line of slope $\frac{1}{\sqrt{\dt}}e^{-\frac{1}{2}}$ to join their turning points smoothly. The explicit expression of this function is given as follows (see (A) of Figure 2).
\begin{equation}\label{EqGauss}
\hat u(x)=\begin{cases}
1,&\mathrm{\ if\ }x\geq  (4-e^{1/2})\sqrt{\dt}:=b,\\
e^{-\frac{|x-b|^2}{2\dt}},&\mathrm{\ if\ }x\in [b-\sqrt{\dt},b),\\
\frac{e^{-1/2}}{\sqrt{\dt}}x+1-2e^{-1/2},&\mathrm{\ if\ }x\in [\sqrt{\dt},b-\sqrt{\dt}),\\
-e^{-\frac{|x|^2}{2\dt}}+1,&\mathrm{\ if\ }x\in [0,\sqrt{\dt}),\\
0,&\mathrm{\ if\ }x< 0.
\end{cases}
\end{equation}
The function $\hat u$ is $\sC^\infty$ everywhere except that $\hat u'''$ is discontinuous at the two turning
 points $\{b- \sqrt\dt, \sqrt\dt\}$ and $\hat u''$ is discontinuous at the two points $\{0, b\}$. 

In the following, we use the notation $B(x, r)$ to denote an open ball centered at $x$ with radius $r$.
\subsubsection{The smoothing}\label{SSSModelSmooth}
We next smoothen the function $\hat u$ by convoluting with a $\sC^\infty$,
nonnegative and compactly supported approximating Dirac-$\boldsymbol\dt$ function denoted by
$\phi_\epsilon$ with supp$\phi_\epsilon(x) \subset (-\epsilon, \epsilon)$ and $\epsilon\ll\dt$,
$$\hat u_\epsilon(x)=\int_{-\infty}^\infty \hat u(t)\phi_\epsilon(x-t)\,dt=\int_{-\infty}^\infty \hat u(x-t)\phi_\epsilon(t)\,dt.$$
The function $\hat u_\epsilon$ has the following properties by taking derivatives in the above expression.
{\it\begin{itemize}
\item $\hat u_\epsilon'\geq 0.$
\item $\hat u_\epsilon''(x)\geq 0$ for $x< b/2$ and $\hat u_\epsilon''(x)\leq 0$ for $x> b/2$.
\end{itemize}}
{\it   Assuming further that $\phi_\epsilon$ is an even function, and $\phi_\epsilon'(x) \geq 0$ for $x < 0$, we get the following

 \ \qquad    $\bullet$ For any given $C > 0$, there is a $\dt_*$ such that for $\epsilon\ll\dt < \dt_*$, in the
interval $x - b \in (−C\dt, 0)$, we have $\hat u''_\epsilon (x) \leq -\dt^{-1}/3$, and in $B(b, \epsilon)$ we have that
  $\hat u_\epsilon''(x)$ is nondecreasing if $\hat u''_\epsilon(x) \geq -\dt^{-1} /3$.}
  
  The third bullet point is less obvious. It will be used to study the concavity property of the profile functions (the third bullet point in Lemma \ref{LmLocation}). Here we give a proof. 
\begin{proof}[Proof of the third bullet point]
We denote $s=x-b$ and compute
$$\hat u''(x) =\begin{cases}
\left(-\frac{1}{\dt}+\frac{s^2}{\dt^2}\right)e^{-s^2/(2\dt)}&s<0\\
0& s>0
\end{cases},\quad \hat u'''(x)=\begin{cases}\frac{s}{\dt}\left(\frac{3}{\dt}-\frac{s^2}{\dt^2}\right)e^{-s^2/(2\dt)}&s<0,\\
\frac{1}{\dt}\boldsymbol\dt(s)&s=0\\
0&s>0
\end{cases}.$$
Since $\hat u''(x)$ is close to $-\dt^{-1}$ for $x-b\in (-C\dt,0)$ and jumps from $-\dt^{-1}$ to $0$ at $b$, and $\hat u_\epsilon''$ is a weighted average of $\hat u''|_{x<0}$ and $0$, we get that $\hat u_\epsilon''(x)<-\dt^{-1}/3$ for $x-b\in (-\dt,-\epsilon)$. To prove the statement, it is enought to show as $\epsilon\ll\dt\to0,$
\begin{equation}\label{Equ3}
\hat u'''_\epsilon(x)=\begin{cases}
\frac{1}{\dt}\phi_\epsilon(s)+O\left(\frac{\epsilon}{\dt^2}\right),& s\in B(0,\epsilon),\\
\left(\frac{1}{\dt}+O\left(\frac{\epsilon}{\dt^2}\right)\right)\phi_\epsilon(s)\geq 0,& s>0,
\end{cases}
\end{equation}
hence $$\hat u_\epsilon''(x)=\int_{b+\epsilon}^x \hat u'''_\epsilon(t)\,dt<-\dt^{-1}/3,\quad \mathrm{for\ }-\epsilon<x-b<0.$$
The first estimate in \eqref{Equ3} follows directly from 
$$\hat u'''_\epsilon(x)=\frac{1}{\dt}\phi_\epsilon(s)+\int_{-\infty}^{0-}\hat u'''(s)\phi_\epsilon(s-t)\,dt,$$
and the second follows from 
$$\left|\int_{-\infty}^{0-}\hat u'''(s)\phi_\epsilon(s-t)\,dt\right|\leq \frac{3\epsilon}{\dt^2}\int_{-\infty}^0\phi_\epsilon(t-s)\,dt\leq \frac{3\epsilon}{\dt^2}\phi_\epsilon(-s)=\frac{3\epsilon}{\dt^2}\phi_\epsilon(s)$$
for $s>0.$
\end{proof}
\subsubsection{Construction of model functions defined on $\R^n$} We define 
$$u(x)=\hat u_\epsilon(x-3\epsilon)$$
so that $u(x)=0$ for $x<\epsilon$ and $u(x)=1$ for $x>b+5\epsilon$. We next define $d=b+6\epsilon$ and $d_s=d/|s|$ for $s\geq 1$. We next introduce 
$$u_s(x)=u(|s|x),\quad v_s(x)=u_s(x+d_s),\quad w_s(x)=u_s((1-d_s)-|x|).$$
Notice that $v_s(x)=1$ for $|s|x\geq -\epsilon$ and $w_s(x)=0$ for $|x|\geq 1-d_s$ and $w_s(x)=1$ for $|x|\leq 1-2d_s$.

We next introduce for $y\in \R^n$ and $|s|\geq 1$
\begin{equation}\label{EqUVW}
U_s(y)=U_1(|s|y)=\prod_{i=1}^n u_s(y_i),\quad V_s(y)=\prod_{i=1}^n v_s(-|y_i|), \quad W_s(y)=w_s(\|y\|/R).
\end{equation}
We also define $W_s=W_1$ for $|s|<1$. The function $W_s$ is rotationally invariant and has nonpositive radial derivatives. 

By definition, we have
\begin{equation}\label{EqU=V}
U_s(y+d_s\mathbf 1)=V_s(y),\quad \mathrm{\ for\ } y_i\leq 0, \ i=1,2\ldots,n,
\end{equation}
where $\mathbf 1:=(1,\ldots,1)\in \R^n$. 
\subsection{Profile functions when the cone $\cC$ is the first quadrant}
In the following we first define the profile functions adapted to the cone $\cC$ defined by $A=\mathrm{id}$ and $p^*=\mathbf 1$. The cone is now the interior of the first quadrant of $\R^n.$
\subsubsection{Profile functions when $s\geq 1$} For $s\geq 1$, we define our profile function 
$$F_s(y)=(c+s)U_s(y).$$
See the upper two curves in Figure 2(B).
\subsubsection{Profile functions when $s\leq 1$}
We define for $s\leq -1$,
$$F_s(y)=((c+|s|)V_s(y-\mathbf 1)-|s|+1/|s|)U_s(y).$$
See the lower two curves in Figure 2(B). The factor $U_s(y)$ is multiplied to guarantee that $F_s(y)=0$ for $y$ close to the boundary of $\cC$.
\subsubsection{Homotopy from $s=1$ to $s=-1$} 
To match the two pieces $s \geq 1$ and $s \leq -1$, we use a homotopic procedure. We first translate the graph of $F_1$ horizontally to match the
graphs of $F_0$ and $F_{-1}$ in the region $y_i \leq 1$, for all $i$. Then we use a linear homotopy
from $F_0$ to $F_{-1}$. Explicitly, the construction is given as follows.
First as $s$ goes from $1$ to $0$, we make a horizontal translation
$$F_s(y)=F_1(y-(1-s)(1-d)\mathbf 1),\quad s\in [0,1).$$
Using \eqref{EqU=V}, we see that $F_{-1}=F_0$ in the region $y_i\leq 1$ for all $i=1,\ldots,n$. Next from $s=0$ to $s=-1$, we use
$$F_s(y)=-sF_{-1}(y)+(1+s)F_0(y),\quad s\in [-1,0).$$
\subsection{ Profile functions adapted to the cone $\cC$} In this section, we build profile
functions adapted to the cone $\cC$ defined by a non degenerate matrix $A$ with $p^*$ in
the interior of $\cC$. The cone $\cC$ is invariant if we multiply each column vector of $A$
by a positive number. We fix these positive numbers using $p^*$ as follows. We first get
a vector $y^* = A^{-1} p^*$ that lies in the interior of the first quadrant since $p^*$ lies in
the interior of $\cC$. We next introduce the matrix $\mathbf A = AY^*$ where $Y^*$ is the diagonal
matrix whose diagonal entries form the vector $y^*$. Now we get $\mathbf A^{-1}p^* =\mathbf 1$. We
denote $y = \mathbf A^{-1} p$ and introduce
                    $$\tilde H_s (p) = F_s (\mathbf A^{-1} p)W_s (\mathbf A^{-1} p),$$ 
where we multiply by the function $W_s$ to make $\tilde H_s (p)$ compactly supported. The family $\tilde H_s$ is constructed to be nondecreasing in $s$, i.e. $\partial_s \tilde H_s \geq 0$, which can be verified by differentiating $\tilde H_s$ directly in each interval of parameter $s$. Notice that at the points $s^* \in \{-1, 0, 1\}$, the family $\tilde H_s$ is not smooth in $s$. We use another smoothing
procedure that is localized in an $\eps$ neighborhood $B(s^*,\eps)$ for some small $\eps\ll \dt$ to
get a family $H_s$ that is $\sC^\infty$ in $s$ with $\partial_s H_s \geq 0$.

We introduce a partition of unity such that one function in the partition of unity
denoted by $\rho_0 (s) \in \sC^\infty$ satisfies $B(s^* , \eps/2)\subset $supp$\rho_0 \subset B(s^* , \eps)$. The smoothing is done by replacing $\rho_0 \partial_s \tilde H_s \geq 0$ by $(\rho_0 \partial_s \tilde H_s )* \phi_\epsilon \geq 0$ where $\phi_\epsilon$ is as before and $\epsilon$ is chosen so small that 
supp$_s(\rho_0\partial_s \tilde H_s)*\phi_\epsilon\subset B(s^* , \eps)$. We denote the resulting function by $H_s$ after integrating the smoothed $\partial_s \tilde H_s$. Notice 
$$\int_{s^*-\eps}^{s^*+\eps}\rho_0(s)\partial_s\tilde H_s\,ds=\int_{s^*-\eps}^{s^*+\eps}\int_{s^*-\eps}^{s^*+\eps}\rho_0(s)\partial_s\tilde H_s\phi_\epsilon(t-s)\,ds\,dt.$$
This implies $H_s = \tilde H_s$ for $s \notin B(s^*,\eps)$.

It is easy to verify that $H_s$ is an exhausting sequence for Hamiltonians in $\widehat{\sH_c} (\cC_{\mathbf A,R}, p^*).$ Namely, for all $H\in \widehat{\sH_c} (\cC_{\mathbf A,R}, p^*),$ there exist $s>s'$ such that $H_{s'}<H<H_s$. It is enough to consider the case $\mathbf A = A = \mathrm{id}$ and $p^* = \mathbf 1$.

\subsection{Location of Morse-Bott manifolds}
In this section, we find the $p_{s}$ satisfying $\frac{\partial H_s}{\partial p}(p_s)=\al$, where $\al$ satisfies the assumption of Theorem \ref{ThmCorner}. 
\begin{Lm}\label{LmLocation}
 Consider $\al \in H_1 (\T^n , \Z)\setminus \{0\}$ as in Theorem \ref{ThmCorner}. Fix a sufficiently
small $\dt > 0$. For each $s \in \R$, there is a unique $p_s^+ \in \cC$ satisfying the following:
   \begin{itemize}
\item $DH_s (p^+_s) = \al$;
\item $D^2 H_s (p_s^+ )$ is negative definite;
\item $H_s (p_s^+ ) -\langle p^+_s , \al\rangle > c - \langle p^* , \al\rangle$.
\end{itemize}
Moreover, any other solution of $H_s(p)=\al$, denoted by $p^-_s$, must satisfy $H_s (p^-_s) -\langle p^-_s ,\al\rangle < 0$.
          \end{Lm}
\begin{proof} We first forget about the cut-off $W_s$ and will study it close to the end of
the proof. We also forget about the smoothing with respect to $s$. The $s$-smoothing
occurs only in a bounded interval $s \in [-2, 2]$. By choosing $\epsilon$ small enough, $H_s$
approximates the non smoothed $\tilde H_s$ as good as we wish in the $\sC^2$ norm in $p$.

{\bf Step 1, existence and uniqueness of $p^+_s$.}

We consider the function $\hat U (y) :=\prod_{i=1}^n\hat u(y_i)=\exp\left(-\frac{\|y-b\mathbf 1\|^2}{2\dt}\right)$ in the region
$$\mathcal D(b, \dt) := \{\|y-b\mathbf 1\| < \dt^{1/2},\ y_i < b,\ i = 1,2,\ldots,n\}.$$
{\bf Claim}: {\it In the region $\mathcal D(b, \dt)$, the map $D\hat U(y)$ is one-to-one and its image covers
$\mathcal C_{\id,r}$ with $r = (e\dt)^{-1/2}$.}
\begin{proof}[Proof of the claim] Consider the level set $\{\hat U (y) = C\}$ where $C \in (e^{-1/2} , 1)$. We get a sphere $\|y-b\mathbf 1\|^2=2\dt(-\ln C)$, whose radius ranges from $0$ to $\sqrt\dt$. Next consider 
\begin{equation}\label{EqDU}
\frac{\partial \hat U}{\partial y_i}=\hat U(y)\cdot(\ln \hat u(y_i))'=-\frac{y_i-b}{\dt}C
\end{equation}
evaluated on each $C$-level set. When $y$ moves on the sphere, the unit vector $\frac{b\mathbf 1-y}{\|y-b\mathbf 1\|}$ attains any vector of the portion of $\mathbb S^{n-1}$ lying in the interior of the first quadrant
since $y_i < b$, $\forall i$. Moreover the modulus
                               $$\|D\hat U\|=C\frac{\|b\mathbf 1-y\|}{\dt}=C\sqrt{\frac{2(-\ln C)}{\dt}}$$ 
ranges from $0$ to $(e\dt)^{-1/2}$ for $C \in (e^{-1/2}, 1)$ and is monotone in $C$ since we have $$(C\sqrt{-\ln C})'=\frac{-2\ln C-1}{2\sqrt{-\ln C}}<0.$$
\end{proof}
The function $U_1 (y)$ in the definition of $F_1$ approximates $\hat U (y)$ in the $\sC^1$ norm if we let $\epsilon\to 0$ in the smoothing. For given $\al$ in the statement of Theorem \ref{ThmCorner}, we choose
some number $r'\gg \al$ independent of the sufficiently small $\dt$. We will show in Step 2 that $U_1 (y)$ is strictly concave in the region $(DU_1 )^{-1} (\cC_{\id,r'} )$. Therefore we get that $DU_1 (y)$ covers $\cC_{\id,r'}$ and is one-to-one in the region $(DU_1 )^{-1} (\cC_{\id,r'})$. The Hamiltonian
equation gives $\dot q = \frac{\partial H_s}{\partial p} = \frac{\partial F_s}{\partial y}  \mathbf A^{-1} =\al$ for $\al \in H_1 (\T^n , \Z)$ in Theorem \ref{ThmCorner}. We get $\beta:=\mathbf A^T\al=\frac{\partial F_s}{\partial y}$ 
lies in the interior of the first quadrant since $\al \in \cC^*$. Notice
that $U_s (y) = U_1 (|s|y)$ for $s \geq 1$, and we have $\frac{\|\beta\|}{|s|(c+|s|)}\ll r'$, hence we can always find
a unique preimage $y_s^+$ of $\beta$ under the map $DF_s (y) : \R^n \to \R^n$ for $y$ in the region $\cD(d_s , \dt_s)$ where $\dt_s = \dt/s^2$, hence a unique $p_s^+ = \mathbf Ay_s^+$.

The same argument applied to the $s < -1$ case shows that a solution can be found
in the region $\mathcal D(1, \dt_s)$. Moreover, $F_s = F_1$,\ $s\in [-1, 0]$ in the same region since
$U_1 (y + d\mathbf 1) = V_1 (y)$ for $y_i \leq 0$ for all $i = 1, 2,\ldots, n$. The cases $s \in [0, 1]$ are only a
translation of the $s = 1$ case. In this way, we find a solution $p^+_s$ for all $s\in \R$.

{\bf Step 2, the Hessian estimate.}

The solution $y_s^+$ that we found in Step 1 lies in a region where $F_s$ is the product of $u_s(y_i)$ up to translations and rescalings and $F_s (y_s^+)$ is close to $\max F_s$. Here we show that $U_1 (y)$ is strictly concave in the region $(DU_1 )^{-1} (\cC_{\id,r'})$ for some large $r'$
independent of $\dt \to 0$. We have the derivatives calculation
$$
\frac{\partial U_1}{\partial y_i}(y)=U_1\frac{u'(y_i)}{u(y_i)},\quad \frac{\partial^2 U_1}{\partial y_i\partial y_j}(y)=U_1(y)\cdot\begin{cases}
\dfrac{u'(y_i)u'(y_j)}{u(y_i)u(y_j)},& i\neq j,\\
\dfrac{u''(y_i)}{u(y_i)},& i=j.
\end{cases}
$$
We have $$U_1 (y)\simeq 1,\quad u(y_i )\simeq 1,\quad 0 < u' (y_i )< 2r',\quad u'' (y_i ) < 0,\quad \forall i.$$ Using the third bullet point in Section \ref{SSSModelSmooth}, and $u'(y_i ) = \int_d^{y_i} u''(x)dx$, we get that $|y_i-d| < 2r' \dt$,
and for each $i$ either $u'' (y_i )< -\dt^{-1}/3$ or $|u'' (y_i )| \geq |u''(x)|$ for $x \in (y_i , d)$ hence in
the latter case $|u' (y_i )| = | \int_d^{y_i}u'' (x)\,dx| \leq  2r'\dt|u'' (y_i )|$. This implies that $D^2 U_1$ is
diagonally dominant and negative definite. By rescaling and linear transformations,
we get that $D^2 H_s (p^+_s)$ is negative definite for all $s \in \R$.

{\bf Step 3, the action estimate.}

For each $s$, we denote by $\hat y_s$ the point satisfying $DF_s (y) = 0$ and closest to $y^+_s$ and denote $\hat p_s = \mathbf A\hat y_s$. We have $\hat y_{s,i} < 1$ for all $i = 1, 2, \ldots, n$ and $U_s (\hat y_s ) = 1$. Hence
we have $H_s (\hat p_s ) = H_s (p^*)$ and $\langle\hat p_s, \al\rangle = \langle \hat y_s , \beta\rangle < \langle\mathbf 1, \beta\rangle = \langle p^*, \al\rangle$. To get the action estimate in the lemma (the third bullet point), it is enough to prove 
$$H_s(p_s^+)-\langle p_s^+,\al\rangle>H_s(\hat p_s)-\langle \hat p_s,\al\rangle> H_s(p^*)-\langle p^*,\al\rangle>c-\langle p^*,\al\rangle.$$
   The latter two inequalities are clear. We work on the first one. We denote $\hat p_s-p_s^+=h$. We then get the following using Taylor expansion
\begin{equation*}
\begin{aligned}
H_s(\hat p_s)-H_s(p_s^+)-(\langle\hat p_s,\al\rangle-\langle p_s^+,\al\rangle)&=H_s(\hat p_s)-H_s(p_s^+)-\langle h,DH_s(p_s^+)\rangle\\
&=\frac{1}{2}\left\langle h,\int_0^1(1-t)D^2H_s(p_s^++th)\,dt\,h \right\rangle<0,
\end{aligned}
\end{equation*}
where the concavity of $D^2 H_s$ along the line segment $p_s^+ + th$ follows from Step 2.

{\bf Step 4, the inequality satisfied by $p^−_s$.}

When $s \geq 1$, we consider possible solutions $p_s^-= \mathbf Ay^-_s$ of $DH_s(p_s)=\al$, other than $p_s^+$ , i.e. $p_s^- \notin \mathcal D(d_s ,\dt_s)$. There must be one entry $y_j$ of $y_s^-$ satisfying $u_s(y_j)<1/2$ since otherwise $\frac{\partial F_s}{\partial y_j}(y_s^-)$ is too large to be $\beta_j$. We then get\begin{equation*}
\begin{aligned}
H_s(p_s^-)-\langle p_s^-,\al\rangle &=F_s(y_s^-)-\langle y_s^-,\beta\rangle<F_s(y_s^-)-y_j\beta_j\\
&\leq F_s(y_s^-)\frac{u_s(y_j)-y_ju_s'(y_j)}{u_s(y_j)}.
\end{aligned}
\end{equation*}
Notice that $F_s(y)>0$ and $u_s(y_j)=\int_0^{y_j} u_s'(t)\,dt<y_j u_s'(y_j)$ since $u_s'$ is monotone and $u''_s\geq 0$. This shows that $H_s(p_s^-)-\langle p_s^-,\al\rangle<0$.

For each $s\in \R$, in the region where $H_s (p) \leq 0$, we always have $H_s (p) -\langle p, \al\rangle < 0,$ 
hence we consider only the region where $H_s (p) > 0$. We pick one such $p_s^-$ for $s = 1$, if any, and consider its continuation to $s < 1$. The horizontal translation for $s \in [0, 1]$ in Section 3.2.3 decreases the action of $p^-_s$ when $s$ decreases. The homotopy for
$s \in [-1, 0]$ does not change $p_s^-$ or $H_s (p^-_s)$ hence does not change the action. The action gets further decreased for $s < -1$.

We next show that as $s$ goes from $1$ to $-\infty$, no new solutions of the equation $DH_s (p) = \al$ with $H_s (p) > 0$ can occur. Consider for instance, $F_s$ with $s <-1$ and
$y_i > 1$ for some $i$. We have $\frac{\partial F_s}{\partial y_i} \leq 0$ for the same $i$ since we have $\frac{\partial V_s}{\partial y_j}(y)\leq  0$ if $y_j \geq 0$, for all $j=1,\ldots,n$. The fact that $\beta_j > 0$, for all $j = 1, \ldots, n$, excludes the possibility of $DF_s (y) = \beta$
having solutions with $y_i > 1$ for some $i$. The same argument applies to all the cases $s \in (-\infty, 0].$

{\bf Step 5, the cut-off $W_s$.}

Finally, let us consider the effect of the cut-off $W_s$. We only need to consider the region where $p\in \cC$ and $\|\mathbf A^{-1} p\|/R \simeq 1$ for large $R$.
           
First consider the case $F_s (y) \geq 0$ where $y =\mathbf A^{-1} p$ lies in the first quadrant, hence $s \geq -1$. We have
$$\frac{\partial H_s}{\partial y}=\frac{\partial F_s}{\partial y}W_s+F_sw_s'\frac{y}{\|y\|R}.$$
Since we have $\|y\|/R \simeq 1$, there must be at least one $j$ such that $y_j > 1$, so that
$u'_s (y_j ) = 0$ hence $\frac{\partial F_s}{\partial y_j}=0$. Moreover we have $w'_s (x) \leq 0$ for $x \geq 0$. Therefore we
have $\frac{\partial H_s}{\partial y_j}\leq 0$ for this $j$, since none of the entries of $F_s w_s'\frac{y}{\|y\|R}$ is positive. However,
we require $\beta_i > 0$ for all $i = 1, 2, \ldots, n$.

The case $F_s \leq 0$ implies $F_s - \langle p, \al\rangle < 0$. The proof is now complete.
\end{proof}

\section{Proof of Theorem \ref{ThmCorner}}\label{sctProof}
In this section, we proof Theorem \ref{ThmCorner} using Lemma \ref{LmLocation} and the machinery set up in Section \ref{sctFloer}.
\subsection{Computation of the action}\label{ssctComp}
We obtain Morse-Bott manifolds corresponding to $p_s^\pm$ denoted by $P_s^\pm$. These Morse-Bott manifolds are Lagrangian tori $\T^n$. Along each periodic orbit $x\subset P_s^\pm$ we evaluate the action
\[\mathcal{A}_{H_s}(x)=\int_0^1 H_s-\langle p,\dot q\rangle \,dt=H_s(p_s^\pm)-\langle p_s^\pm,\al\rangle.\]

\subsection{Proof of the main theorem }\label{ssctProof}
In this section, we prove Theorem \ref{ThmCorner}.
We split the proof into three steps.

{\bf Step 1.} 
{\it Suppose $a\in[0, c-\langle p^*, \al\rangle]$. Then we have $ \iS^{[a,\infty)}(\cC_{\mathbf A,R};\alpha) \cong
  H_*(\T^n;\Z_2) $. Moreover, for all $s\in \R$, the homomorphism
  $$
  \pi_s:\iS^{[a,\infty)}(\cC_{\mathbf A,R};\alpha) \to
  \FH^{[a,\infty)}(H_s;\alpha)
  $$
  is an isomorphism.  }

 Notice that component-wise $$y_{s,i}^+=(\mathbf A^{-1}p_{s}^+)_i< 1=(\mathbf A^{-1} p^*)_i\quad\mathrm{and\ }(\mathbf A^T\al)_i>0,\quad \forall \ s\in\R,\ \forall\ i.$$ Hence
\[\langle p^+_s,\al\rangle =\langle y_s^+,\mathbf A^T\al\rangle<\langle \mathbf 1,\mathbf A^T\al\rangle=\langle p^*,\al\rangle.\]
This means that when $c-\langle p^+_s,\al\rangle> c-\langle p^*,\al\rangle\geq a,$ we have 
\begin{equation}\label{EqSqueezea}
\mathcal A_{H_s}(P_s^-)<0\leq a\leq c-\langle p^*,\al\rangle< \mathcal A_{H_s}(P_s^+),\quad \forall\ s\in \R
\end{equation}
when $a$ satisfies $0\leq a\le c-\langle p^*,\al\rangle.$
Hence, by Theorem~\ref{ThmMB}, $\FH^{[a,\infty)}(H_s;\alpha)
\cong H_*(\T^n;\Z_2)$ since the Morse-Bott manifold $P_s^+$ is a torus, and by Proposition \ref{PropHomotopy} the
monotone homomorphism $\sigma_{H_{s_1}H_{s_0}}$ in \eqref{EqHomo}
is an isomorphism. We now apply Lemma~\ref{LmExhaust} to complete Step 1.



{\bf Step 2.}
{\it Suppose $a\in[0, c-\langle p^*,\al\rangle].$ Then $ \dS^{[a,\infty);c}(\cC_{\mathbf A,R},p^*;\alpha) \cong
  H_*(\T^n;\Z_2).  $ Moreover, for all $s \in\R$ the homomorphism
  $$
  \iota_s:\FH^{[a,\infty)}(H_s;\alpha) \to
  \dS^{[a,\infty);c}(\cC_{\mathbf A,R},p^*;\alpha)
  $$
  is an isomorphism.  }
  
The same argument as the Step 1 with the help of the action computation in Lemma \ref{LmLocation} and
Lemma \ref{LmExhaust} gives us Step 2.

{\bf Step 3.}
{\it Suppose $a\in[0,  c-\langle p^*,\al\rangle],$ Then the homomorphism in Proposition \ref{PropFactor}
  $$
  T^{[a,\infty);c}_\alpha: \iS^{[a,\infty)}(\cC_{\mathbf A,R};\alpha) \to
  \dS^{[a,\infty);c}(\cC_{\mathbf A,R},p^*;\alpha)
  $$
  is an isomorphism.}
  
  By Proposition \ref{PropFactor}, we have $T^{[a,\infty);c}_\alpha=\iota_s\circ\pi_s$ for all $s\in \R$. Then Step 3 follows from Step 1 and 2. 
 
According to the definition of $\widehat{C}(\cC_{\mathbf A,R},p^*;\al,a)$ in \eqref{EqCapacityHat}, we get that for $a\in [0,c-\langle p^*,\al\rangle]$ \[\widehat{C}(\cC_{\mathbf A,R},p^*;\al,a)\leq a+\langle p^*,\al\rangle.\] Applying Proposition \ref{PropCap}, we get \[C(\cC_{\mathbf A,R},p^*;\al,a)\leq \widehat{C}(\cC_{\mathbf A,R},p^*;\al,a)\leq a+\langle p^*,\al\rangle<\infty.\] This implies that periodic orbits exist for all $H\in \sH_c(\cC_{\mathbf A,R},p^*)$. To complete the proof of Theorem \ref{ThmCorner}, it is enough to choose the $R$ in the definition of $\cC_{\mathbf A,R}$ to be large enough so that supp$H\subset \cC_{\mathbf A,R}.$

\section{Proof of Theorem \ref{ThmDense} and \ref{ThmArnold}}\label{sctProofs}
\subsection{Dense existence}\label{ssctDense}
In this section we prove Theorem \ref{ThmDense}. The argument follows that of Theorem 3.4.1 of \cite{BPS}. 

We show that for each  $a,b$ satisfying $\min_q H(p^*,q)>b>a>0$, there exists $s\in (a,b)$ such that the level set $\{H=s\}$ carries a closed orbit in the class $\al$. Define a smooth function $\sigma:\R\to\R$ with the following properties:
\begin{itemize}
\item $\sigma(r)=0$,\ for $r\leq 0$,
\item $\sigma(r)=1$, for $r\geq 1$,
\item $\sigma'(r)>0$, for $0<r<1$.
\end{itemize}
Picking a constant $c\geq \langle p^*,\al\rangle$ and define $F:=c\sigma \left(\frac{H-a}{b-a}\right)$, which can be verified to satisfy the assumption of Theorem \ref{ThmCorner}.
We apply Theorem \ref{ThmCorner} to get that $F$ has a 1-periodic orbit $x$ in the class $\al$ lying on a level set of $\{F=\rho\}$ where $\rho\in (0,c)$. Since $c\sigma:\ (a,b)\to(0,c)$ is injective, there exists $s\in (a,b)$ such that $\{F=\rho\}=\{H=s\}$, hence $x$ lies on the level set $\{H=s\}.$
\subsection{Arnold's problem}\label{ssct:ThmArnold}
In this section, we prove Theorem \ref{ThmArnold}.

We consider the cone $\cC$ determined by the matrix $A=\left[\begin{array}{cc}1&1\\
-1&1\end{array}\right]$. 
In this case, we want to find periodic orbits in a homology class $\al\in \cC^*\cap H_1(\T^2,\Z)$. To get the other homology classes in the statement of the theorem, it is enough to rotate the cone $\cC$ by $\pi/2,\pi,3\pi/2$.  In all the four cases, the cone $\cC$ is the same as its dual cone $\cC^*$. Denote by $\displaystyle M:=-\min_{q\in \T^2}V(q)\geq 0$. We can always find $p^*\in \cC$ and $a,b $ such that \begin{equation}\label{EqArnoldIneq}
H(p^*,q)= \frac{(p^*_1)^2}{2}-\frac{(p^*_2)^2}{2}+V(q)\geq \frac{(p^*_1)^2}{2}-\frac{(p^*_2)^2}{2}-M>b>a>\max_q V(q)=0.
\end{equation}
We further choose $p^*$ such that $c=\langle p^*,\al\rangle> 0$ and define a Hamiltonian function using $\sigma$ in Section \ref{ssctDense} \[F(p,q)=\begin{cases}c\cdot \sigma\left(\dfrac{H(p,q)-a}{b-a}\right)\cdot W_1(p),\quad &p\in \cC,\\
0,& p\in \R^2\setminus \cC,
\end{cases}\]
where $W_1(p)$ is the cut-off function introduced in \eqref{EqUVW} with $s=1$ where $R$ can be chosen as large as we wish. We get that $F$ satisfies the assumption of Theorem \ref{ThmCorner} using \eqref{EqArnoldIneq} since $H(p^*,q)>b,$ and $H(p,q)|_{p\in\partial \cC}=V(q)\leq 0<a$.

We apply Theorem \ref{ThmCorner} to $F$ to get a $1$-periodic orbit in the homology class $\al$ with period one. Let us assume for a moment that the periodic orbit is not created by $W_1\neq 1$. Namely, $\Vert p\Vert/R$ is not close to 1. 

We get a periodic orbit on the energy level $\{H=s\}$ where $s\in (a,b)$. Since $b>a$ can be arbitrary numbers satisfying \eqref{EqArnoldIneq}, we also get dense existence. Namely, there exists a dense subset $S_\al\subset\left(0,\frac{(p^*_1)^2}{2}-\frac{(p^*_2)^2}{2}-M\right)$, such that for each $s\in S_\al$, the energy level $\{H=s\}$ contains a periodic orbit with homology class $\al$.  The argument can be done for any $p^*\in \cC$ satisfying \eqref{EqArnoldIneq}, so we get dense existence in the set of energy levels $(0,\infty).$ 

Finally, we show that the periodic orbit is not created by $W_1\neq 1$. We assume $\Vert p\Vert/R\simeq 1$. We only need to consider $p\in \cC$, since $F(p,q)=0$ when $p\in \R^2\setminus \cC$. We have the Hamiltonian equations
\begin{equation}\label{EqFinal}
\begin{aligned}
&\dot p=-\frac{\partial F}{\partial q}=\dfrac{-c}{b-a}\sigma'\cdot W_1\cdot \frac{\partial V}{\partial q},\\
&\dot q=\frac{\partial F}{\partial p}=\dfrac{c}{b-a}\sigma'\cdot W_1\cdot \frac{\partial H}{\partial p}+c\sigma\cdot w_1'\cdot \dfrac{p}{\Vert p\Vert R}.
\end{aligned}
\end{equation}
Assume first that $\frac{p_1^2}{2}-\frac{p_2^2}{2}>b+M$, then $H(p,q)>b$ so that $\sigma\left(\frac{H(p,q)-a}{b-a}\right)=1$ and $\sigma'=0$. So we get $\dot p=0,$ and $\dot q=c\sigma\cdot w_1'\cdot \frac{p}{\Vert p\Vert R}$. For $R$ large enough, $\Vert\dot q\Vert$ can be made as small as we wish. On the other hand, $\Vert\al\Vert$ is bounded away from zero for $\al \in \cC^*\cap H_1(\T^2,\Z)$. Hence no periodic orbit in homology $\al$ exists in this case. Notice that once a periodic orbit is known to intersect the region $\{\frac{p_1^2}{2}-\frac{p_2^2}{2}>b+M\},$ it in fact always stays there because of $\dot p=0$ and the periodicity. 

It remains to consider a periodic orbit with $\frac{p_1^2}{2}-\frac{p_2^2}{2}\leq b+M$ during time 1. When $\Vert p\Vert/R\simeq 1$, since $\dot p$ is bounded, we must have for all the time either \[|p_1-p_2|\leq 4(b+M)/R,\ |p_1+p_2|\geq R/2\quad \mathrm{or}\quad |p_1+p_2|\leq 4(b+M)/R,\ |p_1-p_2|\geq R/2.\] In the $\dot q$ equation of \eqref{EqFinal}, the factor $\frac{c}{b-a}\sigma'\cdot W_1$ in front of $\frac{\partial H}{\partial p}=(p_1,-p_2)^T$ is bounded, and the second term $c\sigma\cdot w_1'\cdot \frac{p}{\Vert p\Vert R}$ has as small norm as we wish by choosing $R$ large enough. 
For large enough $R$, either $p_1+p_2$ or $p_1-p_2$ is close to zero. This implies that either $\dot q_1-\dot q_2$ or $\dot q_1+\dot q_2$ is close to zero. However, since we assume $\al\in \cC^*\cap H_1(\T^2,\Z)$ (remember that $\cC^*$ is open), we have that both $|\al_1+\al_2|$ and $|\al_1-\al_2|$ are bounded from below by $1$. Hence a 1-periodic orbit of $F$ with $\frac{p_1^2}{2}-\frac{p_2^2}{2}\leq b+M$ and $\Vert p\Vert/R\simeq 1$ cannot have homology class $\al$ for $R$ large enough. This completes the proof. 
\section*{Acknowledgment}
I would like to thank Prof. L. Polterovich for introducing me to the problem, valuable suggestions and constant encouragements.
I would also thank Prof. A. Wilkinson for reading the manuscript carefully and important remarks which led to the dense existence theorem and the discovery of Theorem \ref{ThmArnold}. My thanks also go to my former coadvisor V. Kaloshin from whom I learnt the problem of Arnold. I would like to thank the referee warmly for his/her patient work on pointing out some mistakes in the previous version and improving the readability of the paper. The work is supported by an NSF grant: DMS-1500897.

\end{document}